\documentclass[11pt]{article}
\usepackage[utf8]{inputenc}
\usepackage{lmodern}
\usepackage{subfiles}
\usepackage{enumitem}
\setenumerate{topsep=6pt,ref={\normalfont(\roman*)},label={\normalfont(\roman*)}, itemsep=0pt} 

\usepackage{amsfonts}
\usepackage{amsthm}
\usepackage{amsmath}
\usepackage{amssymb}
\usepackage{amscd}
\usepackage{mathrsfs}
\usepackage{mathtools}
\usepackage{bbm}
\usepackage{esint}

\usepackage[margin=3cm]{geometry}
\usepackage{setspace}
\usepackage{indentfirst}
\usepackage{graphicx}
\usepackage{graphics}
\usepackage{lscape}
\usepackage{pgf,tikz}
\usepackage{tikz-cd}
\usepackage{color}
\usepackage{pict2e}
\usepackage{epic}
\usepackage{epstopdf}
\usepackage{titlesec, titlefoot}
\titleformat{\section}[block]{\Large\bfseries\filcenter}{\thesection}{1em}{}
\usepackage{commath}
\usepackage{float}
\usepackage{caption}
\usepackage{etoolbox}
\usepackage[affil-it]{authblk}
\usepackage{combelow}

\usepackage[hidelinks, bookmarksdepth=3]{hyperref}
\hypersetup{bookmarksopen=true} 
\usepackage{hypcap}

\graphicspath{{./Pictures/}}
\allowdisplaybreaks

\expandafter\def\expandafter\normalsize\expandafter{%
\normalsize
\setlength\abovedisplayskip{6pt}
\setlength\belowdisplayskip{6pt}
\setlength\abovedisplayshortskip{6pt}
\setlength\belowdisplayshortskip{6pt}
}

\theoremstyle{plain}

\renewcommand*\thesection{\arabic{section}}
\numberwithin{equation}{section} 

\newtheorem{theorem}{Theorem}[section]
\newtheorem{lemma}[theorem]{Lemma}
\newtheorem*{lemma*}{Lemma}
\newtheorem{proposition}[theorem]{Proposition}
\newtheorem{corollary}[theorem]{Corollary}
\newtheorem{conjecture}[theorem]{Conjecture}

\theoremstyle{definition}

\expandafter\let\expandafter\oldproof\csname\string\proof\endcsname
\let\oldendproof\endproof
\renewenvironment{proof}[1][\proofname]{%
\oldproof[\upshape \bfseries #1]%
}{\oldendproof}

\makeatletter
\def\@makechapterhead#1{%
\vspace*{50\p@}%
{\parindent \z@ \raggedright \normalfont
\interlinepenalty\@M
\Huge\bfseries  \thechapter.\quad #1\par\nobreak
\vskip 40\p@
}}
\makeatother

\newcommand{\paren}[1]{\left({#1}\right)}
\newcommand{\hodge}{{*}}

\DeclareMathOperator{\Mod}{Mod}
\DeclareMathOperator{\cp}{Cap}
\DeclareMathOperator{\M}{M}

\DeclareMathOperator{\dVol}{dVol}
\DeclareMathOperator{\ddiv}{div}

\def \a{\alpha}
\def \R {\mathbb{R}}
\def \Z {\mathbb{Z}}
\def \C{\mathbb{C}}

\def \N{\mathbb{N}}
\def \D{\textup{D}}

\def \e{\varepsilon}
\def \d{\,\textup{d}}
\def \n{\nabla}
\def \exc{\backslash}
\def \p{\partial}
\def \mc{\mathcal}
\def \mb{\mathbb}
\def \ms{\mathscr}

\def \tp{\textup}

\def \id{\textup{id}}
\def \loc{\textup{loc}}

\begin{document}

	\title{\textbf{On the optimal conformal capacity of linked curves}}
	
	\author[1]{{\Large Andr\'e Guerra}}
	\author[2]{{\Large Eden Prywes}}
	
	\affil[1]{\small Institute for Theoretical Studies,  ETH Zürich,  Zürich,  Switzerland
	\protect\\
	{\tt{andre.guerra@eth-its.ethz.ch}} \ }
	
	\affil[2]{\small Department of Mathematics, Princeton University, Princeton, USA
	\protect\\
	{\tt{eprywes@princeton.edu}} \ }
	
	\date{}
	
	\maketitle

	\begin{abstract}
We investigate the following optimization problem: what is the least possible conformal capacity of a  pair of  linked curves in $\mb S^3$? A natural conjecture, due to Gehring, Martin and Palka, is that the optimal value is attained by the standard Hopf link. We prove that this is the case under the assumption that each component of the link lies on a different side of a conformal image of the Clifford torus.  In particular,  this shows that the Hopf link is a local minimizer in a strong sense. 
	\end{abstract}
	
	\unmarkedfntext{
	\hspace{-0.74cm}
	\emph{Acknowledgments.} AG acknowledges the support of Dr. Max R\"ossler, the Walter Haefner Foundation and the ETH Z\"urich Foundation.  EP was partially supported by the NSF grant RTG-DMS-1502424.  EP would also like to thank Mario Bonk for discussions on this paper.
	}
	
	\section{Introduction}
	
	\subsection{Optimal configurations of linked curves}\label{sec:links}	
	
Given two  connected, compact sets $C_0,C_1\subset \mb S^3$, we define their relative \textit{conformal capacity} through
$$\cp_3(C_0,C_1) = \inf_{u} \int_{\mb S^3} |\tp d u|^3 \dVol,$$
where the infimum is taken over all absolutely continuous functions $u\colon \mb S^3\to [0,1]$ with $u=0$ in $C_0$ and $u=1$ in $C_1$, 
and $\mb S^3=\{(z_1,z_2)\in \C^2: |z_1|^2+|z_2|^2=1\}.$ As the name indicates, the quantity $\cp_3(C_0,C_1)$ is \textit{conformally invariant}.
We call $R=\mb S^3\exc (C_0\cup C_1)$ the \textit{condenser} corresponding to $C_0,C_1$. 


A classical problem in conformal geometry is to find condensers which are extremal for the conformal capacity. For instance, the rings of Gr\"otzsch and Teichm\"uller are well-known examples of extremal condensers \cite[\S 5.1.3]{Gehring2017}.   In this paper, we study configurations of linked curves with the least possible conformal capacity. 
To be precise,  we say that two curves $C_0,C_1$ are \textit{linked} if none of them is contractible in the complement of the other.

It is reasonable to expect that a pair of linked curves that is extremal for the conformal capacity should enjoy strong symmetry properties. One such symmetric configuration is given by the standard Hopf link, which consists of the two circles
	$$
	H_0=\{(z_1,0)\in \mb \C^2:|z_1|=1\} , \qquad
	H_1 =  \{(0,z_2)\in \mb \C^2:|z_2|= 1 \}.
	$$
	In particular, we have the following  natural conjecture \cite[page 315]{Gehring2017}:
	
	\begin{conjecture}[Gehring--Martin--Palka]\label{conj:GMP}
	Let $C_0,C_1\subset \mb S^3$ be linked curves. Then
		\begin{equation}
	\label{eq:Hopfoptimal}
 \cp_3(C_0,C_1)\geq   \cp_3(H_0,H_1) = \frac{16\pi^3}{\Gamma(1/4)^4},
	\end{equation}
	 where $\Gamma(t)$ is the Gamma function.
	\end{conjecture}
In this paper we prove that Conjecture \ref{conj:GMP} holds for linked curves satisfying a certain geometric condition. In order to state our result,  let us denote by $\mb T$ the standard Clifford torus:
	$$\mb T= \{(z_1,z_2)\in \C^2: |z_1|^2=|z_2|^2=1/2\}.$$
	A pair of curves $C_0,C_1\subset \mb S^3$ is \textit{separated by a conformal Clifford torus} if there is a conformal map $\varphi\colon \mb S^3\to\mb S^3$ such that
	$$C_0 \text{ and } C_1 \text{ lie on different components of } \mb S^3\exc \varphi(\mb T).$$
	Note that, in particular, the Hopf link is separated by a conformal Clifford torus.
	We prove the following theorem:
	
	\begin{theorem}\label{thm:introhopf}
	Let $C_0,C_1\subset \mb S^3$ be  linked curves which are separated by a conformal Clifford torus. Then	\eqref{eq:Hopfoptimal} holds.

	\end{theorem}

	As an immediate consequence, the Hopf link is a local minimizer of the conformal capacity among linked curves, in the topology induced by the Hausdorff distance.
		
	\begin{corollary}\label{cor:links}
There is an explicit $\e>0$ such that, if $C_0,C_1$ is a pair of linked curves with $\tp{dist}(C_0,H_0)\leq \e$  and $\tp{dist}(C_1,H_1)\leq \e$,
	then \eqref{eq:Hopfoptimal} holds.
	\end{corollary}

	We now make a few comments about the proof of Theorem \ref{thm:introhopf}. A key difficulty of the problem under consideration is that the standard symmetrization methods are not directly applicable and these are the only general tools available for proving optimality of condensers, see e.g., \cite{Gehring2017,Gehring1961,Sarvas1973} as well as the recent book \cite{Dubinin2014}.  Broadly speaking, in a symmetrization argument one shows that the conformal capacity does not increase  if one symmetrizes the condenser with respect to either a linear space or a sphere. 
The method of moving planes  \cite{Alessandrini1992},  for instance, can be seen as a symmetrization method.  
The main difficulty in proving Conjecture \ref{conj:GMP} is that, given two linked curves, it does not seem possible to symmetrize them while preserving the linking property.
	
	The assumption that the linked curves are separated by a conformal Clifford torus, however, allows us to perform a symmetrization argument in order to prove Theorem \ref{thm:introhopf}. The main observation is that the conformal map
$$\varphi\colon \mb S^3\to \mb S^3, \quad (z_1,z_2)\mapsto (z_2,z_1)$$	
is a ``reflection'' on the Clifford torus. 
This allows us to reduce to the case where $C_1=\varphi(C_0)$ and, in this case, we can show that 
$$\cp_3(C_0,C_1)= \frac 1 2 \cp_3(C_0,\mb T).$$ 
While $C_0$ can still be a very irregular curve, $\mb T$ is extremely symmetric, and hence it can be used as a starting point for a symmetrization argument.
	
	We note that even the statement that the Hopf link is a local minimizer for the conformal capacity, as in Corollary \ref{cor:links},  is not obvious and does not follow directly from variational considerations. Indeed,  it is not difficult to see that the conformal capacity on $\mb S^3$ is not differentiable on curves, since curves have codimension two.  This should be contrasted with the more familiar problem of minimizing the capacity between compact hypersurfaces in $\mb R^3$. 
In this case,  variational arguments yield local optimality of concentric spheres in a $C^{2,\a}$-topology, see \cite[\S 2]{DePhilippis2021} or \cite{Mukoseeva2021}. We refer the reader to \cite{Henrot2005} for a more systematic discussion of variational methods in Shape Optimization.
	
		It seems that completely new ideas are required to remove the structural assumption on the competitors in Theorem \ref{thm:introhopf}.  Besides symmetrization methods, another powerful technique to prove symmetry of optimal condensers is to establish a maximum principle for a suitable nonlinear quantity of the capacity function; this quantity is sometimes referred to as a \textit{$P$-function}. This method was introduced by Weinberger \cite{Weinberger1971} and, in the context of capacity problems, it was first used by Payne and Philippin \cite{Payne1979,Payne1991,Philippin1989}. The $P$-function is very powerful in settings where the optimizer is a spherically symmetric condenser \cite{Agostiniani2022,Agostiniani2020a,Garofalo1999} but, to the best of our knowledge, it has never been successfully applied outside spherical symmetry.	

	\subsection{An overview of other related problems and results}
	
In this short subsection  we briefly review other problems previously studied in the literature which, at least  superficially, appear to be related to Conjecture \ref{conj:GMP}. 
	
In \cite{Agol2015}, 	Agol--Marques--Neves showed that the Hopf link is a minimizer of the conformally-invariant M\"obius cross energy  \cite{Freedman1994}. Their proof relies on the important fact that the M\"obius energy of a link is at least the generalized area of the image of the Gauss map of the link. The authors can then use powerful tools from minimal surface theory to conclude \cite{Marques2014}.  

	Gehring \cite{Gehring1971}, and then later Freedman--He \cite{Freedman1991,Freedman1991a},  also considered a problem closely related to Conjecture \ref{conj:GMP}. They looked for configurations of linked solid tori which are optimal for the conformal capacity. Indeed,  there is a natural notion of conformal capacity for a solid torus $T$. We define
$$\cp_3(T) = \inf_{u} \int_{T} |\tp{d} u|^3 \dVol,$$
where the infimum is taken over absolutely continuous functions $u\colon T \to \R \exc \Z$ with degree one. One can then consider the optimization problem 
	$$\sup_T \min\{\cp_3(T), \cp_3(T^*)\},$$
	where $T$ runs over all solid tori embedded in $\mb S^3$ and $T^*=\overline{\mb S^3\exc T}$ is the torus dual to $T$.  
 Freedman--He conjectured that the optimal configuration is when $T$ and $T^*$ are the closures of the two components of $\mb S^3\exc \mb T$:
\begin{conjecture}[Freedman--He]
For all solid tori $T\subset \mb S^3$ we have
$$ \min\{ \cp_3(T), \cp_3(T^*) \}\leq \cp_3(\mb T_\tp{sol}) =\frac{\sqrt{2}-1}{2\pi},$$
where $\mb T_\tp{sol} = \{(z_1,z_2)\in \C^2 : |z_1|^2\leq |z_2|^2 = 1/2\}$ is the solid Clifford torus. 
\end{conjecture}	
	Earlier,  Gehring \cite{Gehring1971} had formulated a stronger conjecture   which,  if proved,  would give precise bounds for $\cp_3(T^*)$ in terms of $\cp_3(T)$,  for all solid tori $T$.  He was able to prove these bounds under the assumption that either $T$ or $T^*$ is conformally equivalent to a rotationally symmetric torus. 


	\subsection{Branching of quasiregular maps and Martio's conjecture}\label{sec:qc}
	
	In this subsection we explain the connection between Conjecture \ref{conj:GMP} and  Martio's conjecture on the absence of branching for quasiregular maps in space with small distortion. Although we do not have any definitive result in this direction, we believe this connection is worth exploring further.
	
	Given a domain $\Omega\subset \R^n$, a map $f\in W^{1,n}_\tp{loc}(\Omega,\R^n)$ is said to be $K$-\textit{quasiregular} if there is a constant $K\geq 1$ such that
	$$|\D f|^n\leq K \det \D f \quad \text{a.e.\ in }\Omega,$$
	where $|A|$ denotes the operator norm of $A\in \R^{n\times n}$.  
	
	It is worth discussing the case $K=1$ in some detail, as $1$-quasiregular maps are simply \textit{generalized conformal maps}. The behavior of $1$-quasiregular maps depends dramatically on the dimension. If $n=2$, a map is $1$-quasiregular if and only if it is holomorphic. If $n\geq 3$, Liouville's theorem asserts a map is $1$-quasiregular if and only if it is conformal.  Therefore, for $n\geq 3$, $1$-quasiregular maps are the restriction of M\"obius transformations to $\Omega$ \cite[\S 5]{Iwaniec2001} and, in particular,  they are either constant or homeomorphic.
	
	A quasiregular homeomorphism is said to be \textit{quasiconformal}. Quasiconformal mappings can be characterized in terms of the conformal capacity  \cite{Gehring2017}. A homeomorphism $f\in W^{1,n}_\tp{loc}(\Omega)$ is $K$-quasiconformal if and only if 
$$\frac{1}{K} \cp_n(C_0,C_1) \leq \cp_n(f C_0,fC_1)\leq K \cp_n(C_0,C_1)$$	
for all continua $C_0,C_1\subset \Omega$.
Here, the capacity is taken relative to $\Omega$.  
In other words, a homeomorphism $f$ is quasiconformal if and only if the conformal capacity is quasi-invariant under $f$.  More generally, a similar characterization holds for quasiregular mappings, provided that $f$ satisfies mild topological assumptions  \cite[6.2 and 7.1]{Martio1969}.  

	In \cite{Martio1971b}, Martio conjectured that there is a strong form of stability in Liouville's theorem.  To state his conjecture precisely note that, for any quasiregular map $f$, one can define the \textit{inner distortion} $K_I(f)$ through $K_I(f)=\sup_{x\in \Omega} K_I(f,x)$, where 
	$$K_I(f,x)= 
	\begin{cases}
	\frac{ |\tp{cof}(\D f(x))|^n }{ (\det \D f(x))^{n-1} }& \text{if } \det \D f(x)>0,\\ 
	1 &\text{otherwise}.
	\end{cases}$$
As usual, we denote by $B_f$  the \textit{branch set} of $f$, i.e., the set of points where $f$ is not a local homeomorphism. 
	\begin{conjecture}[Martio]\label{conj:Martio}
	Let $n\geq 3$ and  $f\colon \Omega\to \R^n$ be non-constant and quasiregular.  If  $K_I(f)<2$, then $B_f=\emptyset$.
	\end{conjecture}
		
	Martio's conjecture remains open and the best general result, due to Rajala \cite{Rajala2005}, asserts that if $K_I(f)<1+\e(n)$ then $B_f=\emptyset$. Here, $\e(n)>0$ is computable but very small.  The conjecture is known to hold under rather strong assumptions on $f$ \cite{Kauranen2021,Tengvall2022,Tengvall2020}. However, branching is not necessarily related to lack of differentiability of $f$, as there are $C^2$ quasiregular maps with branching \cite{Bonk2004a,Kaufman2005}. Conjecture \ref{conj:Martio} is also known to hold under regularity assumptions on $B_f$ \cite{Gutlyanskii2000,Martio1971b}.  
	
	

If $f$ is quasiregular then the set $f(\Omega)\exc f(B_f)$ is not simply connected, see e.g.\ \cite[\S III, Lemma 5.2]{Rickman1993}, and so one is naturally led to consider curve families linked with $B_f$.  
Let us consider the simplest possible scenario: we take $n=3$ and assume that 
\begin{equation}
\label{eq:line}
B_f\text{ is a topological line. }
\end{equation}
We emphasize that Martio's conjecture remains open even under these strong assumptions.  Consider  the winding map $w\colon \R^3\to \R^3$, defined in cylindrical coordinates through 
	$$w\colon (r,\theta,z)\mapsto (r,2\theta,z).$$
It is easy to verify that
$$K_I(w) = 2 \text{ a.e.\ in } \R^3, \qquad B_w = \{(x,y,z)\in \R^3: x=y=0\},$$
so in particular \eqref{eq:line} holds. The map $w$ also shows that the bound $K_I<2$ in Conjecture \ref{conj:Martio} cannot be improved.  Note that a stereographic projection maps the Hopf link to
	$$\tilde H_0 = B_w,  \qquad \tilde H_1=\{(x,y,z)\in \R^3: z=0, x^2 + y^2=1\}.$$
	Clearly $\tilde H_0$ and $\tilde H_1$ are linked curves that are kept invariant by $w$. 
	
	It is reasonable to expect that, among maps satisfying \eqref{eq:line}, the winding map is extremal.  Assumption \eqref{eq:line} implies that
\begin{equation}
\label{eq:conjugate}
f=\varphi\circ w\circ \psi, \qquad \text{for some } \varphi, \psi \text{ are quasiconformal},
\end{equation}
see \cite[Theorem 4.1]{Church1960} or \cite{Luisto2021,Martio1971} for more general results in this direction. It is natural to conjecture that $K_I(w) \leq K_I(f)$ for all  $f$  as in \eqref{eq:conjugate},
see  \cite[Question 6.1]{Tengvall2022}.  This problem appears to be closely related to Conjecture \ref{conj:GMP}.


\subsection{Outline}
Section \ref{sec:prelims} is a preliminary section and there we gather several useful facts that will be used throughout the paper. We define the relevant topological notions and in Section \ref{sec:cap} we state some basic facts about the $p$-capacity. In Section \ref{sec:pathmod} we introduce the path and surface moduli, which are technically more convenient to work with than the $p$-capacity and are essentially equivalent quantities. In Section \ref{sec:S3} we briefly review the relevant geometry of $\mb S^3$ and in particular we calculate the conformal capacity of the Hopf link.  Finally, Section \ref{sec:proof} contains the proof of Theorem \ref{thm:introhopf}.

	\section{Preliminaries}\label{sec:prelims}
	
	Throughout this section, $\mc M$ denotes a smooth $n$-dimensional Riemannian manifold with Riemannian metric $g$. In practice, we will always consider the cases where either $\mc M=\R^n$ or $\mc M=\mb S^n$, equipped with the standard metrics.  We denote by $\dVol$ and $\mathscr H^{n-1}$ the induced volume form and Hausdorff measure, respectively, and by $\hodge$ and $\tp d^*$ the Hodge star and the codifferential. All other differential operators, such as $\n$ and $\ddiv$, are defined with respect to $g$. 
	
	We now recall some topological notions.
	A \textit{continuum} $C\subset \mc M$ is a connected compact subset which is \textit{non-degenerate} in the sense that it contains more than one point. A \textit{curve} is the continuous image of the unit interval and typically we  identify the map with its image.  In addition, as for continua, we assume that each curve contains more than one point. 
	
	A set $E\subset \mc M$ \textit{separates} $C_0,C_1\subset \mc M$ if $E$ is disjoint from $C_0\cup C_1$ and no connected component of $\mb S^n\backslash E$ contains points from both $C_0$ and $C_1$.   Two curves $C_0,C_1$ are \textit{linked} in $\mb S^3$ provided that $C_0$ is not null-homotopic in $\mb S^3\exc C_1$ and that $C_1$ is not null-homotopic in $\mb S^3\exc C_0$.
	
	\subsection{Condensers and Capacity}\label{sec:cap}
	
%
	
	Throughout this paper we assume that $1<p<\infty$. Given two continua $C_0,C_1\subset \mc M$, we define their $p$-capacity through
	\begin{equation}
	\label{eq:defcap}
	\cp_p(C_0,C_1)= \inf_{u\in \mc A(C_0,C_1)} \int_{\mc M} |\tp{d} u |^p \dVol,
	\end{equation}
	where $\mc A(C_0,C_1)$ is the class of absolutely continuous functions $u\colon \mc M \to [0,1]$ such that $u=0$ in $C_0$ and $u=1$ in $C_1$.  The $n$-capacity is also known as the  \textit{conformal capacity}.
	
	The infimum in \eqref{eq:defcap} is attained by a unique function $u\in W^{1,p}(\mc M,[0,1])$, which we call the \textit{capacity function} of the condenser $(C_0,C_1)$. This function is $p$-harmonic, i.e.
	\begin{equation}
	\label{eq:pharmonic}
	\d^* (|\tp{d} u|^{p-2} \d u)=0,
	\end{equation}
	and therefore it enjoys the following basic regularity property:
	
	\begin{lemma}\label{lemma:reg}
	Let $C_0,C_1\subset \mc M$ be continua and let $u$ be the corresponding capacity function.  There is $\a>0$ such that $u\in C^{1,\a}_\loc(\mb S^3 \exc (C_0\cup C_1))$. 
	\end{lemma}
	

	The following lemma gives another way of calculating the capacity of a condenser.  For the sake of completeness we give the simple proof here.
	
	\begin{lemma}
	\label{lemma:capacitybdry}
Let $C_0,C_1\subset \mc M$ be continua and let $u$ be the corresponding capacity function, which we assume satisfies $\inf_{U_i} |\n u|>0$ for some neighborhoods $U_i$ of $C_i$.  If $t$ is a regular value of $u$, then
	$$\cp_p(C_0,C_1) = \int_{\{u=t\}} |\n u|^{p-1} \d \mathscr H^{n-1}.$$
		\end{lemma}
	
	\begin{proof}
	The proof is a simple application of  the divergence theorem, and we use the fact that, by \eqref{eq:pharmonic}, the vector field $|\n u|^{p-2} \n u$ is divergence free.  
So,  whenever $0<t_0<t<t_1<1$ are regular values of $u$, we have
	\begin{equation}
	\label{eq:indept}
	\int_{\{u=t_0\}} |\n u|^{p-2} \p_{\nu} u \d \mathscr H^{n-1} + \int_{\{u=t_1\}} |\n u|^{p-2} \p_{\nu} u \d \mathscr H^{n-1} = 0, 
	\end{equation}
	where $\nu$ is the outer unit normal to $\{t_1<u<t_2\}$.  We have $\p_{\nu} u = |\n u|$ on $\{u=t_1\}$ and $\p_{\nu} u = -|\n u|$ on $\{u=t_0\}$,  thus \eqref{eq:indept} shows that 
	$$t\mapsto \int_{\{u=t\}} |\n u|^{p-1} \d\mathscr H^{n-1} \tp{ does not depend on } t,$$
	provided that $t$ is a regular value. We now compute
	\begin{align*}
	\int_{\{t_0<u<t_1\}} |\n u|^{p} \dVol & = 
	\int_{\{u=t_1\}} u |\n u|^{p-2} \p_\nu u \d \mathscr H^{n-1} -\int_{\{u=t_0\}} u |\n u|^{p-2} \p_\nu u \d \mathscr H^{n-1} \\
	& = (t_1-t_0) \int_{\{u=t\}} |\n u|^{p-1} \d \mathscr H^{n-1},
	\end{align*}
	and the conclusion follows by sending $t_1\nearrow 1$ and $ t_0\searrow 0$.
	\end{proof}

	\subsection{Path and Surface Modulus}\label{sec:pathmod}

	Instead of operating directly with the conformal capacity, in most of this paper it will  be  technically convenient to work with the conformal path modulus. In our setting, the two quantities are equivalent, see already Lemma \ref{lemma:cap=mod} below.  The conformal modulus was introduced by Ahlfors and Beurling \cite{Ahlfors1950} as a powerful tool to study planar quasiconformal mappings. It was extended to modulus of surfaces and measures by Fuglede \cite{Fuglede1957}.  
	
	Given a collection $\Gamma$ of paths on $\mc M$,   the \textit{path $p$-modulus} of $\Gamma$ is defined as
	$$\Mod_p(\Gamma)= \inf_{\rho\in \mc A(\Gamma)} \int_{\mc M} \rho^p \dVol,$$
	where $\mc A(\Gamma)$ is the class of measurable functions $\rho\colon \mc M\to [0,+\infty]$ such that
	$$\int_{\gamma} \rho \d \mathscr H^1\geq 1 \tp{ for all } \gamma\in \Gamma.$$
	We say that functions in $\mc A(\Gamma)$ are \textit{admissible} for $\Gamma$.

	%
	%

	For the reader's convenience, 	we list some basic properties of the modulus.  Let $\{\Gamma_i\}_{i \in \mb N}$ be path families.  Following \cite{Gehring2017}, we say that $\{\Gamma_i\}_{i \in \mb N}$ are \textit{separated} if there is an associated family of pairwise disjoint Borel sets $E_i \subset \mc M$ so that
	\begin{align*}
	\int_\gamma \chi_{\mc M \setminus E_i} \d\ms H^1 = 0,
	\end{align*}
	whenever $\gamma \in \Gamma_i$ is locally rectifiable. The proof of the next lemma can be found in 	 \cite[Lemma 4.2.1 and Theorem 4.2.9]{Gehring2017}.
	
	\begin{lemma}
	\label{lemma:modprops}
	The $p$-modulus  is an outer measure on path families:
	\begin{enumerate}
	\item $\Mod_p(\emptyset)=0$,
	\item if $\Gamma'\subseteq \Gamma$ then $\Mod_p(\Gamma')\le \Mod_p(\Gamma)$,
	\item given curve families $(\Gamma_j)_{j\in \N}$, we have $\Mod_p\paren{\bigcup_{j=1}^\infty \Gamma_i} \le \sum_{i=1}^\infty \Mod_p(\Gamma_j)$.
	\end{enumerate}		
	The $p$-modulus has an additional property:
	\begin{enumerate}[resume]
	\item\label{it:separate}   given separated curve families $(\Gamma_j)_{j\in \mb N}$,  we have $				\Mod_p\paren{\bigcup_{j=1}^\infty \Gamma_j} = \sum_{j=1}^\infty \Mod_p(\Gamma_j).$
	\end{enumerate}
	\end{lemma}

			Given a pair of continua $C_0,C_1\subset \mc M$, let $\Delta(C_0,C_1)$ denote the family of paths that connect $C_0$ and $C_1$.	
%
%
%
%
	When $p=n$, the connecting modulus is a conformal invariant which is equivalent to the conformal capacity of the corresponding condenser \cite[Theorem 5.2.3]{Gehring2017}. In fact, this equivalence  holds for all $p$, see {\cite[Theorem 7.31]{Heinonen2001}.
	\begin{lemma}\label{lemma:cap=mod}
	Let $C_0,C_1\subset \mc M$ be continua.  Then
	$$\cp_p(C_0,C_1)=\Mod_p(\Delta(C_0,C_1)).$$
	\end{lemma}
	
The infimum in the definition of the modulus is attained by a unique density. The optimal density is $\rho= |\nabla u|$, where $u$ is the capacity function of the condenser $(C_0,C_1)$. Conversely, given an optimal density $\rho$, we can recover the capacity function through
	$$u(x) = \inf_{\gamma_x} \int_{\gamma_x} \rho \d \ms H^1,$$
	where $\gamma_x$ is any curve connecting $x$ to $C_0$ in $\mb S^3\exc(C_0\cup C_1)$.
	
	In addition to the path modulus, we will consider the surface modulus.   We refer the reader to \cite{Aikawa1999, Freedman1991, Gehring1962b, Ziemer1967} and \cite{Kangasniemi2022} for a general  discussion of this concept.  Following \cite[\S 3]{Ziemer1967},  let $\Sigma(C_0,C_1)$ denote the class of smooth hypersurfaces which separate $C_0,C_1$ in $\mc M$. 
	For each set $\sigma \in \Sigma(C_0,C_1)$, we can define a measure $\mu_\sigma$ as
	\begin{align*}
	\mu_\sigma(A) = \mathscr H^{n-1}(A\cap \sigma \exc (C_0\cup C_1)),
	\end{align*}
	where $A$ is any $\mathscr H^{n-1}$-measurable set.  The \textit{surface $p$-modulus} of $\Sigma(C_0,C_1)$ is defined as 
	$$\Mod_p(\Sigma(C_0,C_1)) = \inf_{\rho \in \mc B(C_0,C_1)} \int_{\mc M} \rho^p\dVol,$$
	where $\mc B(C_0,C_1)$ is the class of measurable functions $\rho \colon \mb S^3 \to [0,+\infty]$ such that
	\begin{align*}
	\int_{\sigma} \rho \d\mu_\sigma \ge 1 \text{ for all } \sigma \in \Sigma(C_0,C_1).
	\end{align*}
	As before, $\rho$ is called \textit{admissible} for the family $\Sigma(C_0,C_1)$.
	
	The next result asserts that the path and surface moduli are dual to each other.
	\begin{theorem}[Duality]\label{thm:duality}
	Let $p \in (1,\infty)$ and let $\frac 1 p + \frac 1 q = 1$.  If $C_0,C_1\subset \mc M$ are continua, then
	$$\Mod_p(\Delta(C_0,C_1))^{1/p} \Mod_{q}(\Sigma(C_0,C_1))^{1/q}=1.$$
	\end{theorem} 
	This theorem was proved in \cite[Theorem 3.10]{Ziemer1967} for $p = n$ and a more general class of separating sets, but the result still holds for $\Sigma(C_0,C_1)$ \cite{Gehring1962b} and for all $p$ \cite{Aikawa1999}.
	
	When studying $\Mod_p(\Sigma(C_0,C_1))$, it will be convenient to use differential forms instead of densities in $\mc B(C_0,C_1)$.  
	Given a family $\Sigma$ of smooth hypersurfaces, we define
	\begin{align*}
	\M_p(\Sigma) = \inf_{\omega \in \mc F(\Sigma)} \int_{\mc M} |\omega|^p\dVol,
	\end{align*}
	where $\mc F(\Sigma)$ is the class of $(n-1)$-forms on $\mc M$ such that
	\begin{align*}
	\int_{\sigma} \omega  \geq 1 \text{ for all } \sigma \in \Sigma.
	\end{align*}
This new notion of modulus coincides with the previous one.
	\begin{theorem}\label{thm:modformdensity}
	Let $p \in (1,\infty)$.  If $C_0, C_1\subset \mc M$ are two continua, then
	\begin{align*}
	\M_p(\Sigma(C_0,C_1)) = \Mod_p(\Sigma(C_0,C_1)).
	\end{align*}
	\end{theorem}
	
	As before, the theorem still holds if $\Sigma(C_0,C_1)$ is replaced with a class of more general separating sets and we refer the reader to the \cite{Aikawa1999,Ziemer1967} and \cite{Freedman1991} for this and similar results. 
	
	\begin{proof}
	By Lemma \ref{lemma:cap=mod} and Theorem \ref{thm:duality} it suffices to show that 
	$$\M_p(\Sigma(C_0,C_1)) = \Mod_q(\Delta(C_0,C_1))^{-p/q} = \cp_q(C_0,C_1)^{-p/q},$$ 
	where $\frac 1 p + \frac 1 q = 1$.  
	We can rewrite $\M_p(\Sigma(C_0,C_1))$ as the infimum over $(n-1)$-forms $\omega$ so that $ \tp{d}\omega = 0$ and 
	\begin{align*}
	\int_{\mc M} \omega \wedge \tp{d} u\ge 1,
	\end{align*}
	for any admissible $u\in \mc A(C_0,C_1)$.
	Indeed, by the coarea formula,
	\begin{align*}
	\int_{\mc M} \omega\wedge  \tp{d}u &= \int_0^1 \int_{u^{-1}(t)} \Big\langle \frac{ \tp{d}u}{| \tp{d}u|},*\omega\Big\rangle \d\mathscr H^{n-1} =  \int_0^1 \int_{u^{-1}(t)}\omega ,
	\end{align*}
	where in the last equality we used the fact that $ \tp{d}u/| \tp{d}u|$ is the normal vector to $u^{-1}(t)$.
	Since $u^{-1}(t)$ is a surface separating $C_0$ and $C_1$, we see that if $\omega \in \mc F(\Sigma(C_0,C_1))$ then  $\int_{\mc M} \omega \wedge \tp du \ge 1$.  Conversely, if we have $\int_{\mc M} \omega \wedge \tp{d} u \ge 1$ for every $u\in \mc A(C_0,C_1)$ then, given $\sigma\in \Sigma(C_0,C_1)$, we can consider a sequence $u_j$ of approximations to the characteristic function of $\sigma$ in order to conclude that $\int_\sigma \omega \ge 1$.
	
	Let $u$ be the unique minimizer for $\cp_q(C_0,C_1)$.  For any closed $(n-1)$-form $\omega$, by H\"older's inequality we have
	\begin{align*}
	1 &\le \int_{\mc M} \omega \wedge \tp du \le \paren{\int_{\mc M} |\omega|^p\dVol}^{1/p}\paren{\int_{\mc M} |\tp du|^q \dVol}^{1/q}
	\end{align*}
	and thus
	\begin{align}\label{eq:dualitylowerbound}
	\cp_q(C_0,C_1)^{-1/q} \le \paren{\int_{\mc M} |\omega|^p \dVol}^{1/p}.
	\end{align}
	We claim that equality is attained by $\omega = \cp_q(C_0,C_1)^{-1}\hodge |\tp du|^{q-2}\tp du$, and this will complete the proof.   To prove the claim,  note that the $q$-harmonicity of $u$ implies that $\tp d\omega = 0$ and since $\tp d^2u = 0$ we have that $\tp d^*( |\omega|^{p-2}\omega) = 0$.  Fix $\sigma \in \Sigma(C_0,C_1)$.  Since $\omega$ is closed,  it has the same integral over homologous surfaces, and so 
	\begin{align*}
	\int_\sigma \omega &= \int_{u^{-1}(t)} \omega = \cp_q(C_0,C_1)^{-1} \int_{u^{-1}(t)}  \hodge |\tp du|^{q-2}\tp d u = \cp_q(C_0,C_1)^{-1} \int_{u^{-1}(t)} |\tp du|^{q-1} = 1,
	\end{align*}
	where the last equality follows from Lemma \ref{lemma:capacitybdry}.  
	The $p$-norm of $\omega$ is
	\begin{align*}
	\int_{\mc M} |\omega|^{p}  \dVol &= \cp_q(C_0,C_1)^{-p}{\int_{\mc M} |\tp d u|^q} \dVol = \cp_q(C_0,C_1)^{1-p}.
	\end{align*}
	This is the equality case for \eqref{eq:dualitylowerbound}, which implies that $\cp_q(C_0,C_1)^{-p/q} = \M_p(\Sigma(C_0,C_1))$.	
	\end{proof}
	
	Similarly to Lemma \ref{lemma:modprops} \ref{it:separate},  we have the following result, which has an identical proof:
	
	\begin{lemma}\label{lemma:mprop}
	If $\Sigma_1$ and $\Sigma_2$ are \textit{separate}, i.e. there are disjoint Borel sets $E_1$ and $E_2$ so that 
	\begin{align*}
	\int_\sigma \chi_{\mc M \setminus E_i} = 0
	\end{align*}
	for all $\sigma \in \Sigma_i$, then 
	\begin{align*}
	\M_p(\Sigma_1 \cup \Sigma_2) = \M_p(\Sigma_1) + \M_p(\Sigma_2).
	\end{align*}
	\end{lemma}
	

	\subsection{Geometry in $\mb S^3$}
	\label{sec:S3}
	Most of our analysis will take place in 
$$\mb S^3=\{(z_1,z_2)\in \C^2:|z_1|^2+|z_2|^2=1\},$$	
	which is conveniently parametrized up to a null set by the \textit{Hopf coordinates}. These are coordinates $(\eta,\xi_1,\xi_2)\in [0,\pi/2]\times [0,2\pi)\times [0,2\pi)$ such that 
	$$
	z_1 = e^{i \xi_1} \cos \eta, \qquad z_2 = e^{i \xi_2} \sin \eta.
	$$
	In these coordinates,  the round metric $g_{\mb S^3}$ on $\mb S^3$ can be expressed as
	$$g_{\mb S^3} = \d \eta^2 + (\sin \eta)^2 \d \xi_1^2 + (\cos \eta)^2 \d \xi_2^2$$ 
	and thus 
	\begin{align*}
	\dVol_{\mb S^3} = \sin \eta \cos \eta \d \eta \wedge\tp{d}\xi_1 \wedge \tp{d}\xi_2.
	\end{align*}
	
	
	The \textit{Clifford torus} is the smooth manifold
	$$\mb T = \Big\{(z_1,z_2)\in \mb S^3:|z_1|^2= \frac{1}{2}, |z_2|^2=\frac{1}{2}\Big\},$$
	which separates the two components $H_0,H_1$ of the Hopf link.
	We note that there is a stereographic projection which maps $\mb T$ to
	\begin{equation}
	\label{eq:stereocliff}
	\widetilde{\mb T}= \left\{((\sqrt 2+\cos s) \cos t,(\sqrt 2 + \cos s)\sin t, \sin s),s,t\in [0,2\pi]\right\}\subset \R^3;
	\end{equation}
	thus $\widetilde{\mb T}$ consists of a tube of radius $1$ around a horizontal circle of radius $\sqrt{2}$. A \textit{conformal Clifford torus} is a torus $T=\varphi(\mb T)$, where $\varphi \colon\mb S^3\to \mb S^3$ is a M\"obius transformation.
	
	%
	
	In the next lemma we calculate the conformal modulus of the Hopf link. 
	
	\begin{lemma}\label{lemma:hopf}
	Let $(H_0,H_1)$ be the Hopf link. Then
	$$\Mod_3(\Delta(H_0,H_1)) =  \frac{16 \pi^3}{\Gamma(\frac 1 4)^4}.$$
	\end{lemma}
	
	\begin{proof}
	By Lemma \ref{lemma:cap=mod} our task is to compute the energy of the 3-harmonic function $u$ with $u=i$ on $H_i$. Since $H_0=\{\eta =0\}$ and $H_1=\{\eta=\frac \pi 2\}$, we make the Ansatz that $u$ should depend only on $\eta$.  With $\tp{d} u = u'(\eta)\d \eta$, we can then compute
	$$*|\tp{d}u|\d u= \cos \eta \sin\eta |u'(\eta)| u'(\eta) \d \xi_1\wedge \tp{d} \xi_2$$
	and hence the 3-harmonic equation \eqref{eq:pharmonic} reduces to the ordinary differential equation
	\begin{equation}
	\label{eq:ODE}
	\frac{\d}{\d \eta}\left(\cos\eta \sin \eta |u'(\eta)| u'(\eta)\right)=0.
	\end{equation}
	Assuming that $u'(\eta)>0$ and $u(0)=0$, the general solution to \eqref{eq:ODE} is 
	$$u(\eta) = c \int_0^\eta \frac{\d \theta}{\sqrt{\cos \theta \sin \theta}}= c\int_0^{(\sin\eta)^2} \frac{\d t}{(1-t)^{3/4} t^{3/4}}.$$
	At $\eta=\frac \pi 2$, this integral can be written in terms of the beta function $B$:
	$$1=u(\pi/2) = \frac c 2 B(1/4,1/4) 
	\quad \implies \quad c=\frac {2 \sqrt \pi}{\Gamma(1/4)^2},$$
	where we used the classical identity
	$$B(1/4,1/4)=\frac{\Gamma(1/4)^2}{\Gamma(1/2)} = \frac{\Gamma(1/4)^2}{\sqrt \pi},$$
	cf. \cite[(2.13)]{Artin1964}.
	Since $u'(\eta)=c(\cos \eta \sin \eta)^{-\frac 1 2}$, we have
	$$\Mod_3(\Delta(H_0,H_1)) = \int_{\{\eta=\frac \pi 4\}} |\tp{d}u|^2 \d \mathscr H^2 = (\sqrt{2}c)^2 \mathscr H^2(\mb T) = 2c^2 \times 2\pi^2 =  \frac{16 \pi^3}{\Gamma(\frac 1 4)^4}.$$
	In the first equality, we used the fact that $\{\eta=\frac \pi 4\}$ is a regular level set of $u$ together with Lemmas \ref{lemma:capacitybdry} and \ref{lemma:cap=mod}.
	\end{proof}

		\section{Proof of Theorem \ref{thm:introhopf}}\label{sec:proof}

	In this section we will prove Theorem \ref{thm:introhopf}.  We begin by solving the much easier problem where one of the curves, say $C_0$,  is assumed to be a circle. In this case, we want to consider suitable symmetrizations of $C_1$.  In order to be able to do so, we need to guarantee that linking is preserved under symmetrization. We begin with a simple version of the topological lemma we require.
	
	\begin{lemma}
	\label{lemma:linked}
	Let $C_0=\{x=y=0\}\subset \R^3$ be the $z$-axis and let $C_1$ be a curve which is linked with it.  For any hyperplane $H$ which contains $C_0$,  let $U^+$ and $U^-$ be the two half-spaces on either side of $H$ and let $\widetilde C_1$ be the compact set obtained by reflecting $C_1\cap U^+$ on $H$, thus
$$\widetilde C_1=(C_1\cap U^+)\cup R(C_1 \cap U^+)$$
where $R$ is the reflection on $H$.	
	 Then $\widetilde C_1$ is still linked with $C_0$.
	\end{lemma}
	
	Note that, since $C_1$	is linked with $C_0$, both $C_1\cap U^+$ and $C_1\cap U^-$ are non-empty,  thus $\widetilde C_1$ is non-empty as well.
	
	\begin{proof}
%
%
%
%
%
Clearly we can assume that $C_0$ and $C_1$ are disjoint, as otherwise there is nothing to prove. 
Let us denote by $\pi\colon \R^3\exc C_0\to P$ the projection $(x,y,z)\mapsto (x,y,0)$, where 
$$P=\{(x,y,0): (x,y)\neq (0,0)\}$$ is the punctured horizontal plane. Note that $P$ is a deformation retract of  $\R^3\exc C_0$, according to $t\mapsto t \pi+(1-t) \id$.   Thus $\pi$ induces an isomorphism of fundamental groups, and any curve $C$ in $\R^3\exc C_0$ is null-homotopic if and only if $\pi(C)$ is null-homotopic in $P$.

By the previous paragraph, $\widetilde{C_1}$ is linked with $C_0$ if and only if $\pi(\widetilde{C_1})$ is linked with $\{0\}$ in $P$.   It is easy to see that any curve $C$ in $P$ is linked with $\{0\}$ if and only if the winding number of $C$ about $0$ is non-zero; equivalently, if we write $C(t)=|C(t)| e^{i \theta(t)}$ for some continuous angle function $\theta\colon [0,1]\to \R$,  $C$ is linked with $\{0\}$ in $P$ if and only if $\theta(1)-\theta(0)\in 2\pi \Z\exc \{0\}$.

Since $\pi(C_1)$ is linked with $\{0\}$ in $P$, at least one component of the set $\pi(C_1)\cap U^+$ picks up a half-integer turn: more precisely, there is an interval $[a,b]\subset [0,1]$ such that $C_1(a),C_1(b) \in H$ and such that the angle function $\theta$ of $\pi(C_1)$ satisfies $\theta(b)-\theta(a) =\pm \pi $.  It then follows that at least one of the connected components of the  reflected configuration $\pi(\widetilde{C_1})$ will have a non-zero winding number.
	\end{proof}
	
	We will in fact need the following more general version of Lemma \ref{lemma:linked}, which has an absolutely identical proof:
	
	\begin{lemma}\label{lemma:linkedgeneral}
	Let $C_0\subset \R^3$ be the $z$-axis and let $C_1$ be a curve which is linked with it. Fix a hyperplane $H$ which contains $C_0$ and $k\in \N$.  	We denote by $H_\theta$ the hyperplane containing $C_0$ and meeting $H$ at an angle $\theta$.  Denote by $U^+$, $U^-$ the two infinite wedges between $H$ and $H_{\pi/k}$, cf.\ Figure \ref{fig}.
		
Let $\widetilde C_1$ be the the compact set obtained by successively reflecting $C_1\cap U^+$ on $H_{\pi/k}$, $H_{2\pi/k}$, \dots, $H_{(k-1)\pi/k}$ and then doing a final reflection on $H_\pi=H$.  Precisely, if $R_\theta$ denotes the reflection on $H_\theta$, we set
$$\widetilde C_1 = \widetilde E_1 \cup R_\pi(\widetilde E_1), \qquad 
\widetilde E_1:=(C_1\cap U^+) \cup \left(\bigcup_{j=1}^{k-1} R_{j\pi/k}\circ \dots \circ R_{\pi/k}(C_1 \cap U^+) \right).$$
Then $\widetilde C_1$ is still linked with $C_0$.
	\end{lemma}
	
	\begin{figure}
	\centering
	\includegraphics[width=0.6\textwidth]{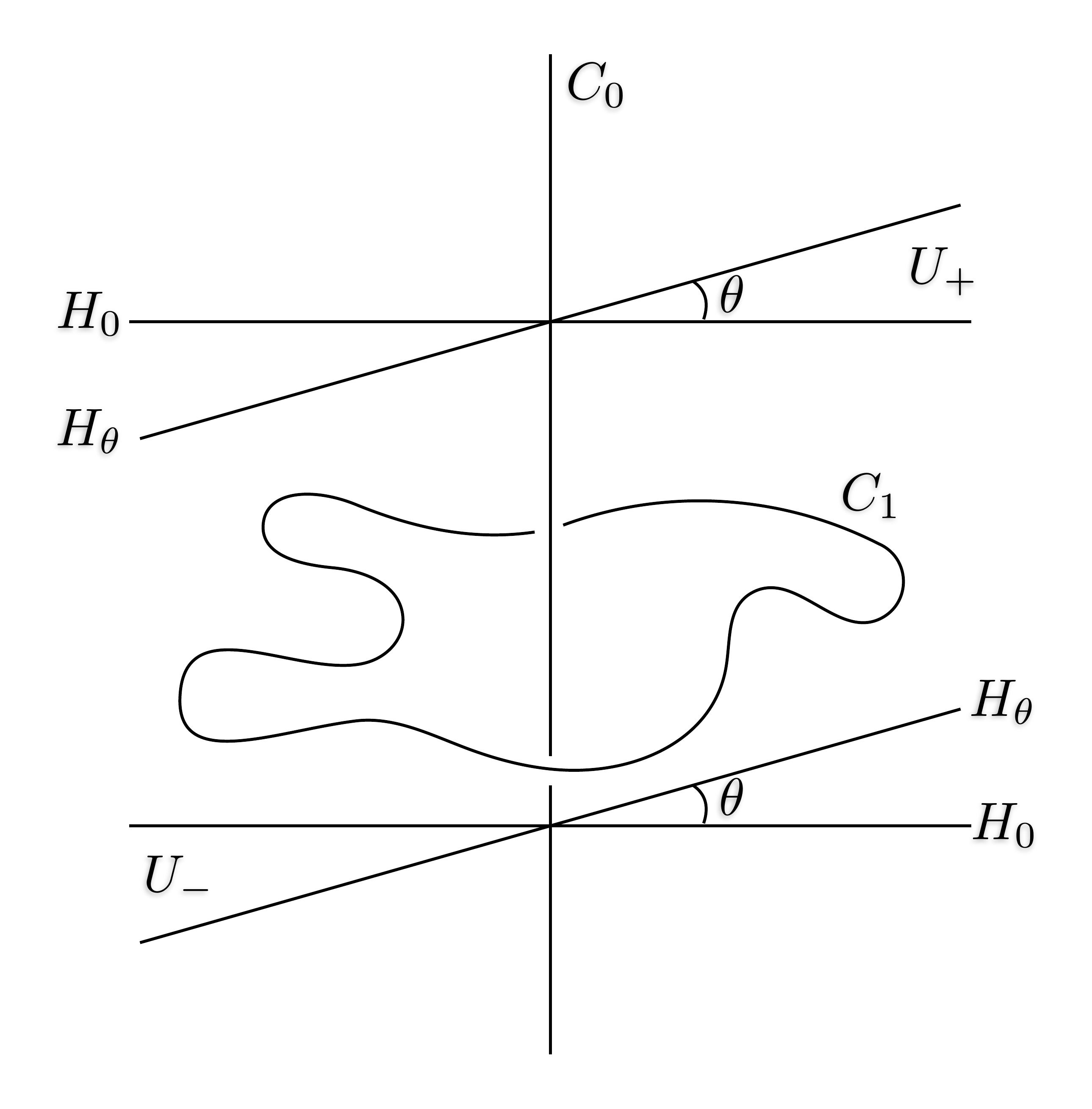}
	\caption{Illustration of the objects in Lemma \ref{lemma:linkedgeneral}}
	\label{fig}
	\end{figure}
	
	Note that we recover Lemma \ref{lemma:linked} from Lemma \ref{lemma:linkedgeneral} by taking $k=1$.	We can now prove that the Hopf link is optimal among all links for which one of the components is a circle.
	
	\begin{proposition}\label{prop:circlecase}
	Let $C_0,C_1\subset \mb S^3$ be two linked curves. If $C_0$ is a circle, then
	$$\Mod_3(\Delta(H_0,H_1))\leq \Mod_3(\Delta(C_0,C_1)).$$
	\end{proposition}
	
	\begin{proof}
	By applying suitable conformal transformations of $\mb S^3$ followed by the stereographic projection, we can assume that $C_0=\{x=y=0\}\subset \R^3$ is the $z$-axis and that $C_1$ is a curve in $\R^3$ which is linked with $C_0$.  
	
	Our main goal is to show that
	\begin{equation}
	\label{eq:axi-sym}
	\inf \left\{\Mod_3(\Delta(C_0,C)): C \text{ is linked with } C_0 \text{ and axisymmetric} \right\}\leq \Mod_3(\Delta(C_0,C_1)).
	\end{equation}
		Let $H$ be a hyperplane that contains $C_0$ and, as in Lemma \ref{lemma:linked}, let  $U^+,U^-$ be the two half-spaces on either side of $H$.   Instead of proving \eqref{eq:axi-sym}, we will first show the weaker claim that
		\begin{equation}
	\label{eq:one-sym}
	\inf \left\{\Mod_3(\Delta(C_0,C)): C \text{ is linked with } C_0 \text{ and symmetric w.r.t.~$H$} \right\}\leq \Mod_3(\Delta(C_0,C_1)).
	\end{equation}

	To prove \eqref{eq:one-sym},  we let  $\rho \in \mc A(C_0,C_1)$ denote the optimal density for $\Mod_3(\Delta(C_0,C_1))$.  Without loss of generality we can assume that
	\begin{align*}
 \int_{U^+} \rho^3\d x\leq\frac{1}{2}\int_{\R^3} \rho^3 \d x = \frac 1 2 \Mod_3(\Delta(C_0,C_1))
	\end{align*}
	If we consider an even reflection of $\rho$ across $H$ we obtain a new density, which we denote by $\tilde \rho$, such that $\tilde \rho = \rho$ on $U^+$. By Lemma \ref{lemma:linked}, we can reflect $C_1\cap U^+$ to generate a new compact set $\widetilde C_1$ that is symmetric with respect to $H$ and that is still linked with $C_0$. 
By monotonicity of the modulus we might as well assume that $\widetilde C_1$ is connected,  in which case it is also a curve. The function $\tilde \rho$ is admissible for $\Delta(C_0,\widetilde C_1)$: if $\tilde \gamma \in \Delta(C_0,\widetilde C_1)$ then, by reflecting on $H$ the parts of $\tilde\gamma$ lying in $U^-$, we obtain a path $ \gamma \subset U^+$ that connects $C_0$ to $ C_1$, and moreover
	\begin{align*}
	\int_\gamma \tilde \rho \d\ms H^1 = \int_{\tilde \gamma} \rho  \d\ms H^1\ge 1.
	\end{align*}
	In addition, we have
	\begin{equation}
	\label{eq:rhodecreases}
	\int_{\R^3} \tilde \rho^3 \d x = 2 \int_{U^+} \rho^3 \d x \leq \int_{\R^3} \rho^3 \d x 
	\end{equation}
	and hence $\Mod_3(\Delta(C_0,\widetilde C_1))\leq \Mod_3(\Delta(C_0,C_1))$,  proving \eqref{eq:one-sym}.
	
	It is now clear how to generalize the previous argument to prove \eqref{eq:axi-sym}. By continuity it suffices to prove that, for any $k\in \N$,  we have
	$$\inf_C \Mod_3(\Delta(C_0,C))\leq \Mod_3(\Delta(C_0,C_1)),$$
	where the infimum is taken over all curves $C$ linked with $C_0$ and invariant under reflections on all hyperplanes $H=H_0,H_{\pi/k}, \dots, H_{(k-1)\pi/k}$,  where we use the notation of Lemma \ref{lemma:linkedgeneral}. This lemma guarantees that we may essentially repeat the construction of the previous paragraph in this more general case.

	By \eqref{eq:axi-sym} we see that, in order to minimize the path modulus, we may assume that $C_1$ is axisymmetric with respect to $C_0$ and hence that it contains a horizontal circle of the form {$S=\{x^2+y^2=r^2,z=c\}$} for some $r>0, c\in \R$.  By the monotonicity of path modulus we see that 
	\begin{equation}
	\label{eq:boundcircle}
	\Mod_3(\Delta(C_0,S)) \leq \Mod_3(\Delta(C_0,C_1)).
	\end{equation}
	But the pair $(C_0,S)$ is conformally equivalent to the Hopf link $(H_0,H_1)$.  Thus the conclusion follows.
	\end{proof}
		
	We now proceed with the proof of Theorem \ref{thm:introhopf}. Let $T$ be a conformal Clifford torus separating $C_0$ and $C_1$. We denote by $U_0$ and $U_1$ the two connected components of 
	$\mb S^3 \setminus T$ satisfying $C_i\subset U_i$, for $i=0,1$, , according to our hypothesis.
	There exists a M\"obius transformation $\varphi \colon \mb S^3 \to \mb S^3$ that maps $U_0$ onto $U_1$. Indeed, if $T=\mb T$ is the standard Clifford torus, then 
	\begin{equation}
	\label{eq:reflection}
	\varphi_{\mb T}(z_1,z_2) =  (z_2,z_1),
	\end{equation}
	where $\mb S^3$ is considered as a subset of $\mb C^2$.
	In the general case,  since $T$ is a conformal Clifford torus, there is a M\"obius transformation $\psi$ such that $T=\psi(\mb T)$, and we may take $\varphi= \psi\circ \varphi_{\mb T}\circ \psi^{-1}$.

	\begin{lemma}\label{lemma:reflectiononT}
	Let $C_0,C_1\subset \mb S^3$ be  linked curves that are separated by a conformal Clifford torus and let $\varphi$ be as above. Then
	\begin{align*}
	\min\big\{\Mod_3(\Delta(C_0,\varphi( C_0))),\Mod_3(\Delta(C_1,\varphi(C_1)))\big\} \le \Mod_3(\Delta(C_0,C_1)).
	\end{align*}
	\end{lemma}
	
	We refer the reader to \cite[Theorem 4.3.3]{Gehring2017} for a similar result with spheres instead of tori.		
	
	\begin{proof}
	For simplicity we write $\widetilde C_i = \varphi(C_i)$ for $i=0,1$.
	By Theorems \ref{thm:duality} and \ref{thm:modformdensity}, it suffices to prove the dual inequality
	\begin{align}
	\label{eq:goalreflection}
	\max\big\{\M_{3/2}(\Sigma(C_0,\widetilde C_0)),\M_{3/2}(\Sigma(C_1,\widetilde C_1))\big\}\ge \M_{3/2}(\Sigma(C_0,C_1)).
	\end{align}
	The family $\Sigma(C_0,C_1)$ can be decomposed as $ \Sigma_0 \cup \Sigma_1 \cup \Sigma_2$, where 
	\begin{align*}
	\Sigma_0 & = \{\sigma \in \Sigma(C_0,C_1) : \sigma \subset \overline{U_0}\}, \\
	\Sigma_1 & = \{\sigma \in \Sigma(C_0,C_1) : \sigma \subset \overline{U_1}\},\\
	\Sigma_2 & = \Sigma(C_0,C_1)\setminus (\Sigma_0 \cup \Sigma_1).
	\end{align*}
	We now split the proof into two steps.
	
	\textbf{Step 1:} We claim that
	\begin{equation}
	\M_{3/2}(\Sigma(C_0,C_1)) = \M_{3/2}(\Sigma_0\cup \Sigma_1).
	\label{eq:exceptional}
	\end{equation}
	Since $\Sigma_0\cup\Sigma_1 \subset   \Sigma(C_0,C_1)$,  by monotonicity of $\M_{3/2}$ we have 
	$\M_{3/2}(\Sigma_0 \cup \Sigma_1) \le \M_{3/2}(\Sigma(C_0,C_1)).$ 
	
	For the opposite inequality,  suppose that $\omega\in \mc F(\Sigma_0\cup \Sigma_1)$ is admissible.  We will first show that $\omega$ is closed or, equivalently, that 
	$$\int_c \omega = 0$$ for any null-homologous surface $c\subset \mb S^3\exc(C_0\cup C_1)$.  It is not difficult to see that $\tp{d}\omega =0$ in $U_i$ for $i=0,1$: indeed, suppose for the sake of contradiction that $c\subset U_i$ is a null-homologous surface with
		\begin{align*}
	\int_{c} \omega = t \ne 0.
	\end{align*}
	If $\sigma \in \Sigma_i$, then $\sigma + kc$ represents a separating surface in $U_i$ for any $k \in \mb Z$.  Therefore, by choosing $k$ so that $|k|$ is sufficiently large and $k$ and $t$ have opposite signs, we have
	\begin{align*}
	\int_{\sigma + kc} \omega = \int_{\sigma}\omega + kt < 1.
	\end{align*}
	This contradicts the assumption that $\omega$ is admissible for $\Sigma_i$. Hence $\tp d \omega=0$ in $U_i$ and it remains to show that $\tp d \omega =0$ on $T$. Pick an arbitrary point $x\in T$ and consider a sphere $S= \p B(x)$ around $x$, where $B(x)$ is a ball with center $x$.  We write
	$$S=A_0\cup A_1, \qquad A_i = (S\cap U_0)\cup (T\cap B(x))$$
	and note that the pieces $T\cap B(x)$ have opposite orientations in $A_0$ and $A_1$. It is now easy to see that
	$$\int_S \omega = \int_{A_0} \omega + \int_{A_1} \omega = 0,$$
	since $\tp{d} \omega =0$ in $U_0\cup U_1$ and the two integrals over $T$ cancel with each other.
	
		
Since $\omega$ is closed, it has the same integral over homologous surfaces.  Since $\mb S^3\exc (C_0\cup C_1)$ deformation retracts to $U_0$,  any separating surface $\sigma \in \Sigma(C_0,C_1)$ is homologous to a separating surface $\sigma_0\subset U_0$, and so
$$\int_\sigma \omega = \int_{\sigma_0} \omega \geq 1,$$
by assumption. Therefore, the form $\omega$ is admissible for $\Sigma(C_0,C_1)$, which implies that 
	$$\M_{3/2}(\Sigma(C_0,C_1)) \ge \M_{3/2}(\Sigma_0\cup \Sigma_1).$$ This proves the claim \eqref{eq:exceptional}.
	
	As the families $\Sigma_0$ and $\Sigma_1$ are separate,  in the sense that any surface in $\Sigma_0$ is contained in $U_0$ and any surface in $\Sigma_1$ is contained in $U_1$,  applying Lemma \ref{lemma:mprop} to \eqref{eq:exceptional} yields
	\begin{align}\label{eq:modsplit}
	\M_{3/2}(\Sigma(C_0,C_1)) = \M_{3/2}(\Sigma_0) + \M_{3/2}(\Sigma_1).
	\end{align}
	
	\textbf{Step 2:} We next would like to show that 
	\begin{align}\label{eq:modsymmetry}
	2 \M_{3/2}(\Sigma_i) =  \M_{3/2}(\Sigma(C_i,\widetilde C_i)).
	\end{align}
	for $i  =0,1$.  Without loss of generality, we choose $i=0$.  
	
	Let $\omega$ be admissible for $\Sigma_0$.  We can define $\tilde \omega$ as $\omega$ on $U_0$ and $\varphi^*\omega$ on $U_1$, and we claim that $\tilde{\omega}$ is admissible for $\Sigma(C_0,\widetilde C_0)$. To see this, fix $\sigma \in \Sigma(C_0,\widetilde C_0)$.   As before, we may decompose $\sigma$ so that $\sigma = \sigma_1 + \sigma_2$, where $\sigma_1 \in \Sigma_0 \cup N_0$ and $\sigma_2 \in \varphi(\Sigma_0 \cup N_0)$; equivalently, we have $\varphi(\sigma_2) \in \Sigma_0 \cup N_0$.  We have that
	\begin{align*}
	\int_{\sigma_1} \tilde{\omega} = \int_{\sigma_1} \omega \quad \text{and} \quad \int_{\sigma_2}\tilde{\omega} = \int_{\varphi(\sigma_2)} \omega.
	\end{align*}
	For the same reason as in Step 1, the integral of $\omega$ on any null-homologous set is $0$ and at least one of $\sigma_1$ and $\varphi(\sigma_2)$ is not null-homologous.  So
	\begin{align*}
	\int_{\sigma} \tilde{\omega} = \int_{\sigma_1} \omega  + \int_{\varphi(\sigma_2)} \omega \ge 1, 
	\end{align*}
	as wished. 
	Additionally,
	\begin{align*}
	\int_{\mb S^3} |\tilde{\omega}|^{3/2} &= 2 \int_{U_0} |\omega|^{3/2} \le 2 \int_{\mb S^3} |\omega|^{3/2},
	\end{align*}
	where in the first equality we used the conformal invariance of the $L^{3/2}$-norm for $2$-forms.  Hence 
	$$2\M_{3/2}(\Sigma_0) \ge \M_{3/2}(\Sigma(C_0,\widetilde C_0)).$$
	
	For the opposite inequality, let $\omega$ be an admissible form for $\Sigma(C_0,\widetilde C_0)$.  We define a new form $\hat\omega = \frac1 2 (\omega + \varphi^*\omega)$ on $U_0$ and $0$ outside $U_0$.  The form $\hat \omega$ is admissible for $\Sigma_0$: given $\sigma \in \Sigma_0$, we have that
	\begin{align*}
	\int_\sigma \hat\omega = \frac{1}{2}\Big(\int_\sigma \omega + \int_{\varphi(\sigma)} \omega\Big) \ge 1.
	\end{align*}
	Additionally, by convexity, 
	\begin{align*}
	\int_{\mb S^3} |\hat \omega|^{3/2} &= \int_{U_0} \abs{\frac{\omega + \varphi^*\omega}{2}}^{3/2} 
	\le \frac{1}{2}\Big( \int_{U_0} |\omega|^{3/2} +\int_{\varphi(U_0)} |\omega|^{3/2} \Big)
	= \frac{1}{2} \int_{\mb S^3} |\omega|^{3/2}.
	\end{align*}
	Therefore we have $\M_{3/2}(\Sigma_0) \le 2 \M_{3/2}(\Sigma(C_0,\widetilde C_0))$ and \eqref{eq:modsymmetry} holds.
	
	Finally, combining \eqref{eq:modsplit} and \eqref{eq:modsymmetry}, we have the identity
	\begin{align*}
	\M_{3/2}(\Sigma(C_0,C_1)) = \frac{\M_{3/2}(\Sigma(C_0,\widetilde C_0)) + \M_{3/2}(\Sigma(C_1,\widetilde C_1))}{2}.
	\end{align*}
	This implies \eqref{eq:goalreflection} and hence the claim of the lemma.
	\end{proof}
	
	\begin{proof}[Proof of Theorem \ref{thm:introhopf}]
	By Lemma \ref{lemma:reflectiononT},  without loss of generality we can assume that $C_1 =\varphi(C_0)$, where $\varphi$ is as in \eqref{eq:reflection}; note that $C_0$ and $\varphi(C_0)$ are linked.  By \eqref{eq:modsymmetry} and Theorems \ref{thm:duality} and \ref{thm:modformdensity},  it follows that
\begin{equation}
\label{eq:halfpicture}
\Mod_3(\Delta(C_0,C_1))=\frac 1 2 \Mod_3(\Delta(C_1,\mb T)).
\end{equation}	
We now apply a symmetrization argument essentially identical to the one from the proof of Proposition \ref{prop:circlecase}.  
	
	More precisely,  after applying a stereographic projection, $\mb T$ is mapped to the torus $\widetilde {\mb T}$ defined in \eqref{eq:stereocliff}.
	In particular, $\widetilde{\mb T}$ is  symmetric with respect to reflections along any hyperplane $H$ that contains the $z$-axis,  and $C_1$ is linked with the $z$-axis, since $C_1$ is linked with $C_0$ and $\mb S^3\backslash C_1$ deformation retracts to $\mb S^1$.  Let $\rho \in \mc A(C_1,\mb T)$  be the optimal density and let $U^+, U^-$ be the two half-spaces on either side of $H$.  We can assume that
	\begin{align*}
	\int_{U^+} \rho^3\d x\leq \frac{1}{2}  \int_{\R^3} \rho^3 \d x .
	\end{align*}
	If we define a new function $\tilde \rho$ through an even reflection of $\rho$ across $H$,  so that $\tilde \rho = \rho$ in $U^+$,  then
	\begin{align}
	 \int_{\R^3} \tilde \rho^3 \d x \le \int_{\R^3} \rho^3 \d x=\Mod_3(\Delta(C_1,\mb T)).
	 \label{eq:lessdensity}
	\end{align}
	We can also reflect $C_1 \cap U^+$ across $H$ to generate a new continua $\widetilde C_1$ which is non-empty, since $C_1$ is linked with the $z$-axis.  We claim that the function $\tilde \rho$ is admissible for $\Delta(\widetilde C_1,\mb T)$.  Indeed, if $\gamma \in \Delta(\widetilde C_1,\mb T)$ then,  by reflecting the parts of $\gamma$ not in $U^+$, we construct a new path $\tilde \gamma \subset U^+$ that connects $C_1$ to $\mb T$, since $\mb T$ itself is invariant under reflections along $H$.  The admissibility of $\tilde \rho$ now follows because
	\begin{align*}
	\int_\gamma \tilde \rho \d \ms H^1= \int_{\tilde \gamma} \rho \d \ms H^1\ge 1.
	\end{align*}
	It then follows from \eqref{eq:lessdensity} that
	$$\Mod_3(\Delta(\widetilde C_1,\mb T)) \leq \Mod_3(\Delta(C_1,\mb T)).$$
	
Arguing exactly as in the proof of Proposition \ref{prop:circlecase} we see that, to find a lower bound for the connecting modulus $\Mod_3(\Delta(C_1,\mb T))$,  we may assume that $C_1$ is symmetric with respect to all hyperplanes containing the $z$-axis; therefore,  we may assume that $C_1$ contains a horizontal circle $S$.  By the monotonicity of the path modulus, we have 
$$\Mod_3(\Delta(S,\mb T))\leq \Mod_3(\Delta(C_1,\mb T)).$$	
Hence, recalling \eqref{eq:halfpicture}, we have have shown that
$$\Mod_3(\Delta(S, \varphi(S))) = \frac 1 2 \Mod_3(\Delta(S,\mb T)) 
\leq \frac 1 2 \Mod_3(\Delta(C_1,\mb T)) = \Mod_3(\Delta(C_0,C_1)).$$
Note that the curves $S$ and $\varphi(S)$ are linked.  Thus,  applying Proposition \ref{prop:circlecase}, we see that an optimal lower bound is provided by the Hopf link.
	\end{proof}

	\let\oldthebibliography\thebibliography
	\let\endoldthebibliography\endthebibliography
	\renewenvironment{thebibliography}[1]{
	\begin{oldthebibliography}{#1}
	\setlength{\itemsep}{0.5pt}
	\setlength{\parskip}{0.5pt}
	}
	{
	\end{oldthebibliography}
	}
	
	{\small
	\bibliographystyle{abbrv-andre}
	\bibliography{/Users/adinis/Documents/Projects/library.bib}

\begin{thebibliography}{10}

\bibitem{Agol2015}
I.~Agol, F.~Marques, and A.~Neves.
\newblock {Min-max theory and the energy of links}.
\newblock {\em J. Am. Math. Soc.}, 29(2):561--578, 2015.

\bibitem{Agostiniani2022}
V.~Agostiniani, M.~Fogagnolo, and L.~Mazzieri.
\newblock {Minkowski Inequalities via Nonlinear Potential Theory}.
\newblock {\em Arch. Ration. Mech. Anal.}, 244(1):51--85, 2022.

\bibitem{Agostiniani2020a}
V.~Agostiniani and L.~Mazzieri.
\newblock {Monotonicity formulas in potential theory}.
\newblock {\em Calc. Var. Partial Differ. Equ.}, 59(1):1--32, 2020.

\bibitem{Ahlfors1950}
L.~Ahlfors and A.~Beurling.
\newblock {Conformal invariants and function-theoretic null-sets}.
\newblock {\em Acta Math.}, 83(1):101--129, 1950.

\bibitem{Aikawa1999}
H.~Aikawa and M.~Ohtsuka.
\newblock {Extremal length of vector measures}.
\newblock {\em Ann. Acad. Sci. Fenn. Math.}, 24(1):61--88, 1999.

\bibitem{Alessandrini1992}
G.~Alessandrini.
\newblock {A symmetry theorem for condensers}.
\newblock {\em Math. Methods Appl. Sci.}, 15(5):315--320, 1992.

\bibitem{Artin1964}
E.~Artin.
\newblock {\em {The gamma function}}.
\newblock Holt, Rinehart and Winston, New York, 1964.

\bibitem{Bonk2004a}
M.~Bonk and J.~Heinonen.
\newblock {Smooth quasiregular mappings with branching}.
\newblock {\em Publ. Math. l'Institut des Hautes Etudes Sci. l'IH{\'{E}}S},
  100(1):153--170, 2004.

\bibitem{Church1960}
P.~T. Church and E.~Hemmingsen.
\newblock {Light open maps on $n$-manifolds}.
\newblock {\em Duke Math. J.}, 27(4):527--536, 1960.

\bibitem{DePhilippis2021}
G.~de~Philippis, M.~Marini, and E.~Mukoseeva.
\newblock {The sharp quantitative isocapacitary inequality}.
\newblock {\em Rev. Mat. Iberoam.}, 37(6):2191--2228, 2021.

\bibitem{Dubinin2014}
V.~N. Dubinin.
\newblock {\em {Condenser Capacities and Symmetrization in Geometric Function
  Theory}}.
\newblock Springer, Basel, 2014.

\bibitem{Freedman1991}
M.~H. Freedman and Z.-X. He.
\newblock {Divergence-Free Fields: Energy and Asymptotic Crossing Number}.
\newblock {\em Ann. Math.}, 134(1):189, 1991.

\bibitem{Freedman1991a}
M.~H. Freedman and Z.~X. He.
\newblock {Links of tori and the energy of incompressible flows}.
\newblock {\em Topology}, 30(2):283--287, 1991.

\bibitem{Freedman1994}
M.~H. Freedman, Z.-X. He, and Z.~Wang.
\newblock {Mobius Energy of Knots and Unknots}.
\newblock {\em Ann. Math.}, 139(1):1, 1994.

\bibitem{Fuglede1957}
B.~Fuglede.
\newblock {Extremal length and functional completion}.
\newblock {\em Acta Math.}, 98(1-4):171--219, 1957.

\bibitem{Garofalo1999}
N.~Garofalo and E.~Sartori.
\newblock {Symmetry in exterior boundary value problems for quasilinear
  elliptic equations via blow-up and a priori estimates}.
\newblock {\em Adv. Differ. Equations}, 4(2):137--161, 1999.

\bibitem{Gehring2017}
F.~Gehring, G.~Martin, and B.~Palka.
\newblock {\em {An Introduction to the Theory of Higher-Dimensional
  Quasiconformal Mappings}}, volume 216 of {\em Mathematical Surveys and
  Monographs}.
\newblock American Mathematical Society, Providence, Rhode Island, 2017.

\bibitem{Gehring1961}
F.~W. Gehring.
\newblock {Symmetrization of Rings in Space}.
\newblock {\em Trans. Am. Math. Soc.}, 101(3):499, 1961.

\bibitem{Gehring1962b}
F.~W. Gehring.
\newblock {Extremal length definitions for the conformal capacity of rings in
  space.}
\newblock {\em Michigan Math. J.}, 9(2), 1962.

\bibitem{Gehring1971}
F.~W. Gehring.
\newblock {Inequalities for condensers, hyperbolic capacity, and extremal
  lengths.}
\newblock {\em Michigan Math. J.}, 18(1), 1971.

\bibitem{Gutlyanskii2000}
V.~Y. Gutlyanskiǐ, O.~Martio, V.~I. Ryazanov, and M.~Vuorinen.
\newblock {Infinitesimal geometry of quasiregular mappings}.
\newblock {\em Ann. Acad. Sci. Fenn. Math.}, 25(1):101--130, 2000.

\bibitem{Heinonen2001}
J.~Heinonen.
\newblock {\em {Lectures on Analysis on Metric Spaces}}.
\newblock Universitext. Springer, New York, NY, 2001.

\bibitem{Henrot2005}
A.~Henrot and M.~Pierre.
\newblock {\em {Variation et optimisation de formes}}, volume~48 of {\em
  Math{\'{e}}matiques \& Applications}.
\newblock Springer, Berlin, Heidelberg, 2005.

\bibitem{Iwaniec2001}
T.~Iwaniec and G.~Martin.
\newblock {\em {Geometric Function Theory and Non-linear Analysis}}.
\newblock Clarendon Press, 2001.

\bibitem{Kangasniemi2022}
I.~Kangasniemi and E.~Prywes.
\newblock {On the Moduli of Lipschitz Homology Classes}.
\newblock {\itshape Preprint}, pages 1--41, 2022,
  \href{http://arxiv.org/abs/2208.14517}{{\ttfamily arXiv:2208.14517}}.

\bibitem{Kaufman2005}
R.~Kaufman, J.~T. Tyson, and J.~M. Wu.
\newblock {Smooth quasiregular maps with branching in $\mathbb R^n$}.
\newblock {\em Publ. Math. l'Institut des Hautes Etudes Sci.}, 101(1):209--241,
  2005.

\bibitem{Kauranen2021}
A.~Kauranen, R.~Luisto, and V.~Tengvall.
\newblock {On BLD-mappings with small distortion}.
\newblock {\em Complex Anal. its Synerg.}, 7(1):1--4, 2021.

\bibitem{Luisto2021}
R.~Luisto and E.~Prywes.
\newblock {Open and discrete maps with piecewise linear branch set images are
  piecewise linear maps}.
\newblock {\em J. London Math. Soc.}, 103(3):1186--1207, 2021.

\bibitem{Marques2014}
F.~C. Marques and A.~Neves.
\newblock {Min-Max theory and the Willmore conjecture}.
\newblock {\em Ann. Math.}, 179(2):683--782, 2014.

\bibitem{Martio1971b}
O.~Martio.
\newblock {A capacity inequality for quasiregular mappings}.
\newblock {\em Ann. Acad. Sci. Fenn. Ser. A I Math.}, 1971:1--18, 1971.

\bibitem{Martio1969}
O.~Martio, S.~Rickman, and J.~V{\"{a}}is{\"{a}}l{\"{a}}.
\newblock {Definitions for quasiregular mappings}.
\newblock {\em Ann. Acad. Sci. Fenn. Ser. A I Math.}, 1969:1--40, 1969.

\bibitem{Martio1971}
O.~Martio, S.~Rickman, and J.~V{\"{a}}is{\"{a}}l{\"{a}}.
\newblock {Topological and metric properties of quasiregular mappings}.
\newblock {\em Ann. Acad. Sci. Fenn. Ser. A I Math.}, 1971:1--31, 1971.

\bibitem{Mukoseeva2021}
E.~Mukoseeva.
\newblock {The sharp quantitative isocapacitary inequality (the case of
  $p$-capacity)}.
\newblock {\em Adv. Calc. Var.}, 16(1):131--162, 2023.

\bibitem{Payne1979}
L.~E. Payne and G.~A. Philippin.
\newblock {On Some Maximum Principles Involving Harmonic Functions and Their
  Derivatives}.
\newblock {\em SIAM J. Math. Anal.}, 10(1):96--104, 1979.

\bibitem{Payne1991}
L.~E. Payne and G.~A. Philippin.
\newblock {On two free boundary problems in potential theory}.
\newblock {\em J. Math. Anal. Appl.}, 161(2):332--342, 1991.

\bibitem{Philippin1989}
G.~A. Philippin and L.~E. Payne.
\newblock {On the conformal capacity problem}.
\newblock {\em Symp. Math.}, 30:119--136, 1989.

\bibitem{Rajala2005}
K.~Rajala.
\newblock {The local homeomorphism property of spatial quasiregular mappings
  with distortion close to one}.
\newblock {\em Geom. Funct. Anal.}, 15(5):1100--1127, 2005.

\bibitem{Rickman1993}
S.~Rickman.
\newblock {\em {Quasiregular Mappings}}.
\newblock Springer, Berlin, Heidelberg, 1993.

\bibitem{Sarvas1973}
J.~Sarvas.
\newblock {Symmetrization of condensers in $n$-space}.
\newblock {\em Ann. Acad. Sci. Fenn. Ser. A I Math.}, 1973:1--44, 1973.

\bibitem{Tengvall2022}
V.~Tengvall.
\newblock {A self-contained proof to Martio's conjecture in the class of
  BLD-maps}.
\newblock {\itshape Preprint}, pages 1--15, 2022,
  \href{http://arxiv.org/abs/2208.07072}{{\ttfamily arXiv:2208.07072}}.

\bibitem{Tengvall2020}
V.~Tengvall.
\newblock {Remarks on Martio's conjecture}.
\newblock {\em Math. Scand.}, 128(3):1--15, 2022.

\bibitem{Weinberger1971}
H.~F. Weinberger.
\newblock {Remark on the preceding paper of Serrin}.
\newblock {\em Arch. Ration. Mech. Anal.}, 43(4):319--320, 1971.

\bibitem{Ziemer1967}
W.~P. Ziemer.
\newblock {Extremal Length and Conformal Capacity}.
\newblock {\em Trans. Am. Math. Soc.}, 126(3):460, 1967.

\end{thebibliography}
	}

\end{document}